\documentclass[a4paper, 11pt]{article}

\usepackage[titletoc,toc,title]{appendix}

\input xy   
\xyoption{all}
\usepackage{amsmath, amsthm, amsfonts, amssymb}
\usepackage{enumerate}
\usepackage{stackrel}
\usepackage{colordvi}
\usepackage{hyperref}
\usepackage[english]{babel}
\usepackage{color}
\usepackage{verbatim}
\usepackage[all]{xy}
\usepackage{graphicx}
\usepackage{mathrsfs}

\DeclareMathAlphabet{\mathcalligra}{T1}{calligra}{m}{n}
\DeclareMathAlphabet{\mathscrmin}{T1}{mathscr}{m}{n}

\numberwithin{equation}{section}
\theoremstyle{plain}
  \begingroup
        \newtheorem{theorem}[equation]{Theorem}
        \newtheorem{lemma}[equation]{Lemma}
        \newtheorem{proposition}[equation]{Proposition}
        \newtheorem{corollary}[equation]{Corollary}

	    \newtheorem{definition}[equation]{Definition}
        \newtheorem{notation}[equation]{Notation}

\endgroup

\usepackage{amsmath,amssymb,amsthm, xspace,a4wide,color}
\theoremstyle{definition}
  \begingroup
        \newtheorem{remark}[equation]{Remark}
        \newtheorem{example}[equation]{Example}

\endgroup

\newcommand{\alg}[1]
{(#1,\overline{#1})}

\newcommand{\cc}{\mathcal}

\newcommand{\ff}{\mathsf}

\newcommand{\A}{\mathcal{A}}
\newcommand{\B}{\mathcal{B}}

\newcommand{\K}{\mathcal{K}}
\newcommand{\Cat}{\mathcal{C}at}

\newcommand{\mr}[1]{\overset {#1} {\longrightarrow}}

\newcommand{\xr}[1]{\xrightarrow {#1}}

\newcommand{\Mr}[1]{\overset {#1} {\Longrightarrow}}

\newcommand{\mrpairviejo}[2]
   {
    \xymatrix@C=5ex@R=2.4ex
            {
             {} \ar@<1.6ex>[r]^{#1} 
	            \ar@<-1.1ex>[r]^{#2} 
	         & {}
            }
   }

\newcommand{\mrpair}[2]
   {
    \xymatrix@C=5ex@R=2.4ex
            {
             {} \ar@<1ex>[r]^{#1} 
	            \ar@<-1ex>[r]_{#2} 
	         & {}
            }
   }

 \newcommand{\mrpairc}[2]
   {
    \xymatrix@C=5ex@R=2.4ex
            {
             {} \ar@<1ex>[r]^{#1} 
	            \ar@<-1ex>[r]|{o}_{#2} 
	         & {}
            }
   }

 \newcommand{\mrpaircc}[2]
   {
    \xymatrix@C=5ex@R=2.4ex
            {
             {} \ar@<1ex>[r]|{o}^{#1} 
	            \ar@<-1ex>[r]|{o}_{#2} 
	         & {}
            }
   }

\newcommand{\mlpair}[2]
   {
    \xymatrix@C=5ex@R=2.4ex
            {
             {} 
              & {} \ar@<1.0ex>[l]_{#2} 
	          \ar@<-1.7ex>[l]_{#1}
            }
    }

\newcommand{\cellrd}[3] 
 {
  \xymatrix@C=7ex@R=2.4ex
         {
          {} \ar@<1.6ex>[r]^{#1} 
             \ar@{}@<-1.3ex>[r]^{\!\! {#2} \, \!\Downarrow}
             \ar@<-1.1ex>[r]_{#3} 
          & {}
         }
 }
 \newcommand{\modif}[3] 
 {
  \xymatrix@C=7ex@R=2.4ex
         {
          {} \ar@<1.6ex>@{=>}[r]^{#1} 
             \ar@{}@<-1.3ex>@{=>}[r]^{\!\! {#2} \, \!\downarrow}
             \ar@<-1.1ex>[r]_{#3} 
          & {}
         }
 }
 \newcommand{\scellrd}[3] 
 {
  \xymatrix@C=4.5ex@R=2.4ex
         {
          {} \ar@<1.6ex>[r]^{#1}
             \ar@{}@<-1.3ex>[r]^{\!\! {#2} \, \!\Downarrow}
             \ar@<-1.1ex>[r]_{#3}
          & {}
         }
}

\newcommand{\cellld}[3] 
 {
  \xymatrix@C=6ex@R=2.4ex
         {
            {} 
          & {} \ar@<1.0ex>[l]^{#3} 
          \ar@{}@<-1.7ex>[l]^{\!\! {#2} \, \!\Downarrow}
	                                 \ar@<-1.7ex>[l]_{#1}
         }
 }

\newcommand{\cellpairrd}[4] 
 {
  \xymatrix@C=10ex@R=2.4ex
         {
          {} \ar@<1.6ex>[r]^{#1} 
             \ar@{}@<-1.3ex>[r]^{\!\! {#2} \, \!\Downarrow 
                                 \;\; {#3} \, \!\Downarrow }
             \ar@<-1.1ex>[r]_{#4} 
          & {}
         }
 }

\newcommand{\cellpairrdcorto}[4] 
 {
  \xymatrix@C=6ex@R=2.4ex
         {
          {} \ar@<1.6ex>[r]^{#1} 
             \ar@{}@<-1.3ex>[r]^{\!\! {#2} \, \!\Downarrow 
                                 \;\; {#3} \, \!\Downarrow }
             \ar@<-1.1ex>[r]_{#4} 
          & {}
         }
 }

\newcommand{\cellpairrdc}[4] 
 {
  \xymatrix@C=10ex@R=2.4ex
         {
          {} \ar@<1.6ex>[r]^{#1} 
             \ar@{}@<-1.3ex>[r]^{\!\! {#2} \, \!\Downarrow 
                                 \;\; {#3} \, \!\Downarrow }
             \ar@<-1.1ex>[r]|{o}_{#4} 
          & {}
         }
 }
 
\newcommand{\cellpairrdcc}[4] 
 {
  \xymatrix@C=10ex@R=2.4ex
         {
          {} \ar@<1.6ex>[r]|{o}^{#1} 
             \ar@{}@<-1.3ex>[r]^{\!\! {#2} \, \!\Downarrow 
                                 \;\; {#3} \, \!\Downarrow }
             \ar@<-1.1ex>[r]|{o}_{#4} 
          & {}
         }
 }
 
      \newcommand{\ssLim}[2]
   {
    \underset{#1}{\underleftarrow{\ff{\sigma s Lim}}}
    \; {#2}
   }   
   
      \newcommand{\ssopLim}[2]
   {
    \underset{#1}{\underleftarrow{\ff{\sigma s op Lim}}}
    \; {#2}
   }

   \newcommand{\sosopLim}[2]
   {
    \underset{#1}{\underleftarrow{\ff{\sigma}  s \ff{opLim}}}
    \; {#2}
   }

\newcommand{\dcell}[1]  
          {
					 \ar@<8pt>@{-}[d]+<-4pt,8pt> 
           \ar@<-8pt>@{-}[d]+<4pt,8pt>
           \ar@{}[d]|{#1}
          }

\newcommand{\dcellb}[1]   
          {
           \ar@<10pt>@{-}[d]+<-5pt,8pt> 
           \ar@<-10pt>@{-}[d]+<5pt,8pt>
           \ar@{}[d]|{#1}
          }

\newcommand{\deq}        
         {
          \ar@{=}[d]
         }
        
\newcommand{\dreq}       
         {
          \ar@{=}[dr]
         }

\newcommand{\dleq}       
         {
          \ar@{=}[dl]
         }

\newcommand{\dccell}[1]    
          {
           \ar@{-}[ld] 
           \ar@{-}[rd] 
           \ar@{}[d]|{#1}  
          }

\newcommand{\dcellbb}[1]   
          {
           \ar@<20pt>@{-}[d]+<-10pt,12pt> 
           \ar@<-20pt>@{-}[d]+<10pt,12pt>
           \ar@{}[d]|{#1}
          } 

\newcommand{\dl}    
          {                        
           \ar@<-2pt>@{-}[d]+<4pt,8pt>
          }

\newcommand{\dr}    
          {                        
           \ar@<2pt>@{-}[d]+<-4pt,8pt> 
          }
\newcommand{\dc}[1]    
          {                        
           \ar@{}[d]|{#1}  
          }
\newcommand{\dcr}[1]    
          {                        
           \ar@{}[dr]|{#1}  
          }

\newcommand{\uccell}[1]      
          { 
           \ar@{-}[ur] 
           \ar@{}[u]|{#1} 
           \ar@{-}[ul] 
          }

\newcommand{\uccellb}[1]     
          { 
           \ar@<-1ex>@{-}[ur] 
           \ar@{}[u]|{#1} 
           \ar@<1ex>@{-}[ul] 
          }

\newcommand{\dcellop}[1]  
          {
					 \ar@<6pt>@{-}[d]+<6pt,8pt> 
           \ar@<-6pt>@{-}[d]+<-6pt,8pt>
           \ar@{}[d]|{#1}
          }

\newcommand{\dcellopb}[1]  
          {
					 \ar@<7pt>@{-}[d]+<7pt,8pt> 
           \ar@<-7pt>@{-}[d]+<-7pt,8pt>
           \ar@{}[d]|{#1}
          }

\newcommand{\did}{\ar@2{-}[d]}

\newcommand{\op}[1]
          {
           \ar@{-}[ld] 
           \ar@{-}[rd] 
           \ar@{}[d]|{#1}  
          }

\newcommand{\cl}[1]
          { 
           \ar@{-}[ur] 
           \ar@{}[u]|{#1} 
           \ar@{-}[ul] 
          }

\newenvironment{acknowledgements}
  {
   \begin{abstract}}
  {\end{abstract}
   }

\begin{document}

\title{
Lifting PIE limits with strict projections}

\author{Szyld M.}

\maketitle

\begin{abstract}
We give a unified direct proof of the lifting of PIE limits to the 2-category of algebras and (pseudo) morphisms, which specifies precisely which of the projections of the lifted limit are strict and detect strictness. 
In the literature, these limits were lifted one by one, so as to keep track of these projections in each case. 
We work in the more general context of weak algebra morphisms, so as to include lax morphisms as well. 
PIE limits are also all simultaneously lifted in this case, provided some specified arrows of the diagram are pseudo morphisms. 
Again, this unifies the previously known lifting of many particular PIE limits, which were also treated separately. 
\end{abstract}

\begin{acknowledgements}
I thank Eduardo J. Dubuc for several 
conversations on the subject of this paper.
\end{acknowledgements}

\section{Introduction} \label{sec:intro}

PIE limits, that is the 2-dimensional limits which can be constructed from products, inserters, and equifiers, have received significant attention, especially since the proof in \cite{K2dim} that they can be lifted to the 2-category of algebras over a monad and (pseudo) morphisms between them. 
In \cite[\S 2]{K2dim} the lifting of some PIE limits is shown as follows: first products, inserters, and equifiers are lifted separately, and then so are several other limits which can be constructed from them, also with independent proofs for each of these. The reason why this is done is that it is not only relevant to know that a limit can be lifted, but there is a further result: some of the projections of the lifted limit are strict, and detect strictness. So, to see which are these specified projections in each case, the authors are forced to either look carefully at the construction of each of these limits from products, inserters, and equifiers (as they do), or to give a separate proof for each limit to be lifted (as it is noted in \cite[Remark 2.8]{K2dim}).

\smallskip

{\em The main result of this article, Theorem \ref{teo:main}, gives as an immediate corollary the lifting of any PIE limit to the 2-category of algebras and (pseudo) morphisms, specifying which of the projections of the lifted limit are strict and detect strictness.}

\smallskip

As it turns out, this family of projections of a PIE limit was already considered in the characterization of PIE limits in terms of their weights given in \cite{PIE}: it is the family corresponding to the initial objects of the connected components of the 1-dimensional category of elements given by the weight. But the fact that these are precisely the strict and strictness-detecting projections of the lifted limit seems to be new, and this is really the key to the lifting result. 

It is now known (\cite[\S 6.4]{LSAdv}) that PIE limits are precisely the limits that can be lifted to all these 2-categories of algebras and pseudo morphisms.
Also, it is shown in op. cit. that PIE weights coincide with the coalgebras for a particular comonad, and that this fact can be used for their lifting. However, as the authors point out themselves, the results so obtained contain no information specifying the strict and strictness-detecting projections. In addition, when considering the case of lax morphisms, the result here is significantly stronger than the one in op. cit.

The proof we give of this theorem consists of an adaptation of the proof of \cite[Th. 5.1]{S.lifting}, and as such it deals with a {\em conical} expression of the PIE limit, this is the notion of (conical) sigma strict limit (denoted $\sigma$-s-limit) which goes back to \cite{Gray} but that we introduce in the present paper with a modern notation. 
$\sigma$-s-limits are 
like the $\sigma$-limits of \cite{DDS1}, but where identities are taken in place 
of isomorphisms in the the structural 2-cells of the $\sigma$-cones. 
In other words, $\sigma$-s-limits are similar to lax limits, but the corresponding notion of cone requires the structural 2-cells corresponding to a fixed family $\Sigma$ of arrows of the indexing 2-category $\A$ to be identities. We refer to $(\A,\Sigma)$ as the indexing pair of the \mbox{$\sigma$-s-limit}. 
It has been long known (\cite{S}) that (conical) $\sigma$-s-limits are just as expressive as weighted 2-limits, but the latter seem to have been preferred in the developments of \mbox{2-dimensional} category theory that came since, much in the spirit of $\cc{C}at$-based 
category theory. A point I want to make is that working directly with 
$\sigma$-s-limits brings a new approach to the theory of 2-limits, as 
it is extensively illustrated in \cite{DDS1} and \cite{S.lifting}.

\smallskip

In Section \ref{sec:PIElimits} we carry this out for the case of PIE limits, by defining what it means for an indexing pair to be PIE. 
Combining the results in \cite{PIE} and \cite{S} mentioned above the definition is quite clear: for an indexing pair to be PIE, the subcategory given by $\Sigma$ should be a disjoint union of categories with initial objects. 
It is relevant however that there are expressions of many PIE limits by indexing pairs that don't come from PIE weights. 
We give examples of PIE $\sigma$-s-limits, and a construction of an arbitrary PIE $\sigma$-s-limit from products, inserters, and equifiers, which is closer to the classical construction of limits from products and equalizers than the one that can be found in \cite{PIE}.
This allows in particular to conclude that the assignations between weights and indexing pairs given in \cite[Th. 14, 15]{S} restrict to PIE weights and PIE indexing pairs.

\smallskip

In Section \ref{sec:PIElimitsofalgebras} we lift PIE limits to the 2-categories of algebras and weak morphisms. These 2-categories were introduced in \cite{S.lifting} in order to deal with the usually considered notions of algebra morphisms simultaneously: we fix a family of 2-cells in which the structural 2-cell of the algebra morphism is required to be. 
Theorem \ref{teo:main} deals with the lifting of an arbitrary PIE limit to these 2-categories, and specifies which of the projections are strict and strict-detecting. 
As a first corollary, we get the result mentioned above that unifies several results in \cite[\S 2]{K2dim}. 
In fact, this yields also a slight strengthening of these results: in our case, the base category is not required to have any other limit than the one we lift. 
I should note that for this case in which every algebra morphism has an invertible structural 2-cell, though the proofs would have been more complicated, it would have sufficed to consider just weighted PIE limits; the only advantage of the approach by indexing pairs is the greater simplicity of the conical shape. However, for the more subtle case of lax morphisms, in which not all PIE limits can be lifted, there is also a mathematical gain that comes with $\sigma$-s-limits.

Consider a diagram 2-functor $\A \mr{} T$-$Alg_{\ell}$ whose limit we want to lift. 
As it is known since \cite{Llax}, usually some of the arrows of the diagram should be (pseudo) algebra morphisms for the limit to be lifted. This is also the case in Theorem \ref{teo:main} for lifting a PIE $\sigma$-$s$-limit, the unique arrow in $\Sigma$ going to each object from the initial object of its connected component is required to be a pseudo algebra morphism. 
 The strength of this hypothesis, however, depends greatly (for a fixed type of limit) on its {\em presentation} as a $\sigma$-s-limit, which at the same time makes it relevant to find different presentations of the same limit. In general, the presentation coming from the construction in \cite{S} gives hypothesis that are too strong when compared to the results in \cite{Llax} and \cite{S.lifting} (basically all the arrows of the diagram are required to be pseudo algebra morphisms in this case). But these results can be deduced from Theorem \ref{teo:main} precisely when we use the approach of working directly with $\sigma$-s-limits, that is when we consider the {\em intrinsic} presentations of these $\sigma$-s-limits that don't come from their expressions as weighted limits.

 \section{PIE $\sigma$-s-limits} \label{sec:PIElimits}

For a basic reference on the subject of limits in 2-category theory see \cite{KElem}. 
We begin by recalling, mainly for fixing the notation we will use in this article, the category of lax cones of a given 2-functor.

\begin{definition}
Let $\A \mr{F} \B$ be a 2-functor, and $E \in \B$. A lax cone (for $F$, with vertex $E$) is given by the following data (which amounts to a lax natural transformation from the constant 2-functor at $E$ to $F$): $\{E \mr{\theta_A} FA\}_{A \in \A}$,  $\{Ff \theta_A \Mr{\theta_f} \theta_B\}_{A\mr{f}B \in \A}$, which is required to satisfy the following equations:

\smallskip

\noindent
\begin{tabular}{rll}
 {\bf LC0}. & For all $A\in \A$, & ${\theta}_{id_A} = id_{\theta_A}$. \\
 {\bf LC1}. & For all $A \mr{f} B \mr{g} C \in \A$, & $\theta_{gf} = \theta_g \circ Fg \theta_f$. \\
 {\bf LC2}. & For all $A \cellrd{f}{\gamma}{g} B \in \A$, & $\theta_f = \theta_g \circ F\gamma \theta_A$. 
\end{tabular}

An op-lax cone is defined analogously but the structural $2$-cells $\theta_f$ are reversed. When it is safe, we omit the prefix ``op", and the evident dual statements.

A morphism of lax cones $\theta \mr{\alpha} \theta'$
(which amounts to a modification between the lax natural transformations) 
 is given by the data $\{\theta_A \Mr{\alpha_A} \theta'_A\}_{A \in \A}$ satisfying:

\smallskip

\noindent
\begin{tabular}{rll}
 {\bf LCM}. & For all $A \mr{f} B \in \A$, & $\quad \quad \; \theta'_f \circ Ff \alpha_A = \alpha_B \circ \theta_f$. 
\end{tabular}

\smallskip

In this way we have a category $Cones_{\ell}(E,F)$.
\end{definition}

The following easy Lemma 
(which could also be stated for general lax natural transformations and modifications and is probably well-known)
will be used in Section \ref{sec:PIElimitsofalgebras}. It allows to {\em modify} a lax cone by a given family of invertible 2-cells. For convenience we give the op-lax version.

\begin{lemma} \label{lema:nuevocono}
Given an op-lax cone $\{E \mr{\theta_A} FA\}_{A \in \A}$, $\{\theta_B \Mr{\theta_f} Ff \theta_A\}_{A\mr{f}B \in \A}$, a family of arrows $\{E \mr{\theta'_A} FA\}_{A \in \A}$, and a family of invertible 2-cells $\{\theta_A \Mr{\alpha_A} \theta'_A\}_{A \in \A}$, the definition $\theta'_f = Ff \alpha_A \circ \theta_f \circ \alpha_B^{-1}$ yields the only possible op-lax cone structure such that the $\alpha_A$ form a modification.
\end{lemma}

\begin{proof}
The equation defining $\theta'_f$ is clearly equivalent to the one in axiom {\bf LCM}. The verification of the {\bf LC0}-{\bf LC2} axioms for $\theta'$ is immediate.
\end{proof}

\begin{notation} \label{not:ASigma}
We consider fixed throughout this section a family $\Sigma$ of arrows of the \mbox{$2$-category} $\A$, closed under composition and containing all the identities. We denote the arrows of $\Sigma$ with a circle: $\xymatrix{\cdot \ar[r]|{\circ} & \cdot}$
\end{notation}

We now describe explicitly the notions of $\sigma$-s-cone and (conical) $\sigma$-s-limit. These are notions originally considered by Gray in \cite{Gray}, but we introduce them here with a notation and approach closer to \cite{DDS1}, \cite{S.lifting}. Note that this corresponds to the case $\Omega = \Omega_s$ in \cite[Def. 2.6]{S.lifting}.

\begin{definition} \label{def:sslim}
Let $\A \mr{F} \B$ be a 2-functor, and $E \in \B$. A $\sigma$-s-cone (for $F$, with vertex $E$) is a lax cone which satisfies the additional equation (note that it implies equation {\bf LC0}):

\smallskip

\noindent
\begin{tabular}{rll}
 {\bf $\sigma$sC}. & For all $f \in \Sigma$, & ${\theta}_{f}$ is the identity 2-cell. \\
\end{tabular}

\smallskip

The category of $\sigma$-s-cones, $Cones_{\sigma}^s(E,F)$, is the full subcategory of $Cones_{\ell}(E,F)$. The (conical) $\sigma$-s-limit of $F$ (with respect to $\Sigma$) is 
universal $\sigma$-s-cone, denoted \mbox{$\{\ssLim{A\in \A}{FA} \mr{\pi_A} FA\}_{A\in \A}$,} \mbox{$\{Ff \pi_{A} \Mr{\pi_f} \pi_B \}_{A\mr{f} B \in \A}$}, in the sense that for each \mbox{$E\in \B$,} post-composition with $\pi$ is an isomorphism of categories
\begin{equation}\label{isoplim}
\; \B(E,\ssLim{A\in \A}{FA}) \mr{\pi_*} Cones_\sigma^s(E,F) 
\end{equation}
We refer to the arrows $\pi_A$, for $A \in \A$, as the projections of the limit.
\end{definition}

As it is usual, we say that the limit $\pi$ satisfies a one-dimensional universal property (every cone $\theta$ factors uniquely as $\pi_* \phi$) and a two-dimensional universal property (every morphism of cones $\theta \mr{\alpha} \theta'$ induces a unique 2-cell $\phi \Mr{\beta} \phi'$. The notion of $\Omega$-compatible limit (\cite[Def. 3.11, Rem. 3.12]{S.lifting}), which was key regarding their  lifting to the 2-categories of algebras, deals with a ``restriction" of this two-dimensional universal property to a family $\Omega$ of 2-cells of $\B$. For our case in Section \ref{sec:PIElimitsofalgebras} we will need a slight modification of this notion:

\begin{definition}
Let $\A \mr{F} \B$ a 2-functor, $\A_0$ a family of objects of $\A$, and $\Omega$ a family of 2-cells of $\B$. 
We say that the limit $\ssLim{A\in \A}{FA}$ is \mbox{$\A_0$-$\Omega$-compatible} if, in the correspondence between morphisms of cones $\theta \mr{\alpha} \theta'$ and 2-cells $\phi \Mr{\beta} \phi'$ given by the 2-dimensional universal property, if the 2-cells $\alpha_{A_0}$ are in $\Omega$ for each $A_0 \in \A_0$, then so is $\beta$.
\end{definition}

Note that, when $\A_0$ consists of all the objects of $\A$, we recover the notion of $\Omega$-compatible. When $\Omega$ consists of all the 2-cells of $\B$, or the invertible ones, or just the identities, then every limit is $\Omega$-compatible (\cite[Rem. 3.13]{S.lifting}). Trivial as this is, it is a fact implicitly used in \cite{K2dim}, \cite{Llax} when lifting limits, and a very particular case is also proved ``by hand" in \cite[Lemma 3.1]{Llax}.

We recall now the fact, mentioned in the introduction, that any 2-limit admits an expression as a $\sigma$-s-limit. 
$\cc{E}l_W$ stands for the 2-category of elements (Grothendieck construction) of $W$, and $\Diamond_W$ is the usual projection. 
For a proof see \cite[Th. 15]{S} or \cite[Prop. 3.18]{S.lifting}. 

\begin{proposition} \label{prop:weightedcomoconical}
Let $\A \mr{W} \Cat$, $\A \mr{F} \B$.
 The weighted 2-limit $\{W,F\}$ is equal to the \mbox{$\sigma$-s-limit} of the $2$-functor $\cc{E}l_W \mr{\Diamond_W} \A \mr{F} \Cat$ (with respect to the family $\Sigma$ of arrows of $\cc{E}l_W$ of the form $(f,id)$), in the sense that the universal properties defining each limit are equivalent.
\qed
\end{proposition}

\begin{remark} \label{rem:weightedcomoopconical}
If we want to express a 2-limit as a (conical) $\sigma$-s-op-limit, we can either do it with a dual proof to the one of the Proposition above, or by using general results relating limits and op-limits (see \cite[Rem 3.7]{S.lifting} and the references therein). In any case, the indexing pair for the $\sigma$-s-op-limit is easily seen to be $(\Gamma_W,\Sigma)$, where $\Gamma_W$ is a dual version of $\cc{E}l_W$ (see \cite[Rem. 2.5.2]{DDS1} for details).
\end{remark}

\begin{example}\label{ex:pielimits}
We give now various examples of $\sigma$-s-limits which can be constructed from products, inserters, and equifiers. They are considered in \cite[\S 2]{K2dim}, \cite{Llax} for their lifting to \mbox{2-categories} of algebra morphisms, and they are described explicitly as weighted 2-limits in \cite[\S 4]{KElem}, \cite{PIE}. For convenience regarding the application to these limits of the results in Section \ref{sec:PIElimitsofalgebras}, we give the $\sigma$-s-oplimit versions, that is, $\ssopLim{}{F}$. It is easy to adapt them for $\sigma$-s-limits.

\smallskip 

\noindent
{\bf 1. Product.} This is the case when the indexing 2-category $\A$ is a set.

\smallskip 

\noindent
{\bf 2. Inserter.} This goes back to \cite[I,7.10 2)]{Gray}: it is given by the diagram $\{ \xymatrix{ A  \ar@<1ex>[r]|{\circ}^{f} \ar@<-1ex>[r]_{g} & B} \} \mr{F} \B $. Details can be found in \cite[Examples 3.15, 3.22, 5.3]{S.lifting}.

\smallskip 

\noindent
{\bf 3. Equifier.} It is given by the diagram $\{ \xymatrix@C=3pc{ A \ar@{}[r]|{\alpha \Downarrow \; \beta \Downarrow } \ar@<1.5ex>[r]|{\circ}^{f} \ar@<-1.5ex>[r]_{g} & B} \} \mr{F} \B $. For details see \cite[Ex. 3.16]{S.lifting}.

\smallskip 

\noindent
{\bf 4. Inverter.} It seems that it has never been considered as a $\sigma$-s-limit. The following is a presentation obtained applying Proposition \ref{prop:weightedcomoconical} to the definition in \cite{KElem}, but a better presentation may also be found. It is convenient to consider a more general case and describe a $\sigma$-s-opcone for a diagram 
$\Big\{ \vcenter{\xymatrix@C=3pc{
A \ar[r]|{\circ}^{f} \ar@<.3ex>@{}[rd]^{\alpha \Downarrow} \ar[rd]|{\circ}_g & B \ar@<.7ex>[d]^{h} \\
& C \ar@<.7ex>[u]^k }} \Big\} \mr{F} \B$, where $kh = id_B$, $hk = id_C$ and $\alpha$ is a 2-cell $f \Mr{} kh$ (there is also the 2-cell $h\alpha$). From the $\sigma$-s-opcone axioms it follows that such a $\theta$ is given by $\theta_A$ and $\theta_h$, with $\theta_k = \theta_A F\alpha$ ({\bf LC2}), and $ Fk\theta_h \circ \theta_k = id$, $Fh \theta_k \circ \theta_h = id$ ({\bf LC1} and {\bf LC0}). The inverter is then obtained when $FB = FC$ and $Fh = Fk = id$.

\smallskip 

\noindent
{\bf 5. Cotensor.} Given a category $\A$ which we consider as a discrete 2-category, and $B \in \B$, consider $\A^{op} \mr{F} \B$ constant at $B$ and $\Sigma$ only the identities.

\smallskip 

\noindent
{\bf 6. Comma object.} This also goes back to \cite[I,7.10 1)]{Gray}: consider the diagram $\{ \xymatrix{ A  \ar[r]^{f} & B & C \ar[l]|{\circ}_{g} } \} \mr{F} \B$. 
\end{example}

In \cite[\S 2, \S 3]{PIE}, PIE limits are characterized in terms of its weights. 
Considering their definition of PIE weight together with Proposition \ref{prop:weightedcomoconical}, we are led to define:

\begin{definition}
We say that a pair $(\A,\Sigma)$ is a PIE indexing pair, or for short that it is PIE, if the 1-subcategory of $\A$ given by all the objects of $\A$ and the arrows of $\Sigma$ satisfies that each of its connected components has an initial object.
\end{definition}

Note that for the two indexing pairs $(\cc{E}l_W,\Sigma)$ and $(\Gamma_W,\Sigma)$ considered in Proposition \ref{prop:weightedcomoconical} and Remark \ref{rem:weightedcomoopconical}, the 1-subcategory considered in this Definition is the same category $Gr$-$ob(W)$ considered in $\cite{PIE}$. Thus, we have

\begin{proposition} \label{prop:xx}
If $W$ is a PIE weight in the sense of \cite{PIE}, then $(\cc{E}l_W,\Sigma)$  and $(\Gamma_W,\Sigma)$ are PIE indexing pairs. \qed
\end{proposition}

An informal idea which is convenient to have in mind is that, given a (PIE) limit $\{W,F\}$, the (PIE) indexing pair $(\cc{E}l_W,\Sigma)$ has, by construction, the same {\em complexity} of $W$, in the sense that the cones for each of the limits in Proposition \ref{prop:weightedcomoconical} consist of the same information but only arranged differently.
There is a positive aspect of the $\sigma$-s-limit which is its conical shape, for which it is easier to have an intuition.
 But furthermore, it seems to be usually the case that the same limit admits a simpler expression as a $\sigma$-s-limit (see Examples 2, 3 and 6 above), different to the one given by Proposition \ref{prop:weightedcomoconical}. In particular, these simpler expressions will allow for stronger results when we apply the lifting theorem of Section \ref{sec:PIElimitsofalgebras} to them. 

\begin{notation}
When $(\A,\Sigma)$ is PIE, we will denote by $\A_0$ the family of the initial objects of each connected component of the 1-subcategory given by $\Sigma$. Given $A \in \A$, we will denote by $\xymatrix{ A_0 \ar[r]|{\circ}^{f_A} & A }$ the unique arrow in $\Sigma$ from the initial object of its connected component.
\end{notation}

\begin{remark}[{cf. \cite[Lemma 2.3]{PIE}}] \label{rem:A0compatible}
If $(\A,\Sigma)$ is PIE, $\A \mr{F} \B$ is a 2-functor, and $\theta$ is a \mbox{$\sigma$-s-cone} for $F$, then the family $\{ \theta_A \}_{A \in \A}$ is completely determined by the subfamily 
$\{ \theta_{A_0} \}_{A_0 \in \A_0}$ (since by {\bf $\sigma$sC} we have $\theta_A = F(f_A) \theta_{A_0}$). Also, a morphism of lax cones $\theta \mr{\alpha} \theta'$ is determined by its components $\alpha_{A_0}$ with $A_0 \in \A_0$ (by {\bf LCM} we have $\alpha_A = F(f_A) \alpha_{A_0}$). It follows that, for any family $\Omega$ of 2-cells of $\B$, the limit $\ssLim{A \in \A}{FA}$ is $\A_0$-$\Omega$-compatible if and only if it is $\Omega$-compatible.
\end{remark}

The previous remark, though immediate, is really significant regarding the construction of PIE limits from products, inserters, and equifiers (it allows to avoid using equalizers, similarly to \cite[Lemma 2.3]{PIE}) and also for their lifting to the 2-categories of algebras. Note that for a PIE $\sigma$-s-limit there is a distinguished family of projections from which the other ones are uniquely constructed; these projections will be precisely the ones that are strict and detect strictness when lifted. It is instructive to consider this Remark in the Examples 1 to 6 above and recover the usual projections of these limits. We also have the following immediate Corollary:

\begin{corollary} \label{coro:projjm}
If $(\A,\Sigma)$ is PIE and $\A \mr{F} \B$ is a 2-functor admitting a $\sigma$-s-limit as in Definition \ref{def:sslim}, then the family $\{ \pi_{A_0} \}_{A_0 \in \A_0}$ of projections is jointly monic. Also, if a pair of 2-cells is equal after
composing with all the $\pi_{A_0}$ (for $A_0 \in \A_0$), then they are equal. \qed
\end{corollary}

We will now give the results for PIE $\sigma$-s-limits analogous to \cite[Prop. 2.1, Th. 2.2]{PIE}, that is an explicit construction of PIE-indexed limits from products, inserters, and equifiers. Though we could deduce these results from the ones of op. cit. using \cite[Th. 14]{S}, we consider a direct proof to be much clearer. 

\begin{proposition} \label{prop:yy}
If $(\A,\Sigma)$ is PIE, and $\A \mr{F} \B$ is a 2-functor, then 
$\ssLim{A \in \A}{FA}$
can be constructed from products, inserters, and equifiers.
\end{proposition}
 
\begin{proof}
The proof has similar ideas to the ones of \cite[Lemmas 2.3, 2.5, 2.6]{PIE}, but the construction in this conical case is closer to the classical construction of limits from products and egalizers, and allows for a much simpler notation.

We consider first the inserter $I$ of the diagram 
$\displaystyle \prod_{A_0 \in \A_0}{FA_0} \xymatrix{ \ar@<1ex>[r]^{\phi_0} \ar@<-1ex>[r]_{\phi_1} & 
} \prod_{A \mr{f} B} FB $ in which $\phi_0$ and $\phi_1$ are induced respectively by the arrows $F(f f_A) \pi_{A_0}$ and $F(f_B) \pi_{B_0}$. Note that a cone of this diagram consists of a family $E \mr{\theta_{A_0}} FA_0$ together with (if we define $\theta_A$ as the composition $F(f_A) \theta_{A_0}$) a family of 2-cells $F(f) \theta_A \Mr{\theta_f} \theta_B$. It is also easy to check that a morphism of two such cones $\theta, \theta'$ is given by a family $\{ \theta_{A_0} \Mr{\alpha_{A_0}} \theta'_{A_0} \}_{A_0 \in \A_0}$ such that if we define $\alpha_A = F(f_A) \alpha_{A_0}$ it satisfies axiom {\bf LCM}. 
We denote by $\theta$ the inserter cone.

We will thus obtain the desired limit as the equifier of a diagram which expresses the equations of a $\sigma$-s-cone, that is the diagram 
$$\displaystyle I \xymatrix@C=3pc{ \ar@<1.5ex>[r]^{\psi_0} \ar@<-1.5ex>[r]_{\psi_1} \ar@{}[r]|{\eta_0 \Downarrow \; \eta_1 \Downarrow} & } 
\prod_{A \xymatrix@C=1.5pc{\ar[r]|{\circ}^f & } B} FB
\; \times
\prod_{A \mr{f} B \mr{g} C} FC
\; \times
\prod_{A \cellrd{f}{\gamma}{g} B} FB
$$

in which the arrows $\psi_i$ and the 2-cells $\eta_i$ ($i=0,1$) are given by the equalities in the axioms {\bf $\sigma$sC}, {\bf LC1} and {\bf LC2}:

\smallskip

\noindent- For each $A \xymatrix@C=1.5pc{\ar[r]|{\circ}^f & } B$ we consider 
$I \xymatrix@C=3pc{ \ar@<1.5ex>[r]^{\theta_B} \ar@<-1.5ex>[r]_{\theta_B} \ar@{}[r]|{id \Downarrow \; \theta_f \Downarrow} & } FB$ (note that by the PIE hypothesis we have $A_0 = B_0$ and $f f_A = f_B$, so $\theta_B = F(f) \theta_A$),

\smallskip

\noindent- For each $A \mr{f} B \mr{g} C$ we consider 
$I \xymatrix@C=5pc{ \ar@<1.5ex>[r]^{F(gf)\theta_A} \ar@<-1.5ex>[r]_{\theta_C} \ar@{}[r]|{\theta_{gf} \Downarrow \; \theta_g \circ Fg \theta_f \Downarrow} & } FC$, and 

\smallskip

\noindent- For each $A \cellrd{f}{\gamma}{g} B \in \A$ we consider
$I \xymatrix@C=5pc{ \ar@<1.5ex>[r]^{F(f) \theta_A} \ar@<-1.5ex>[r]_{\theta_B} \ar@{}[r]|{\theta_{f} \Downarrow \; \theta_g \circ F\gamma \theta_A \Downarrow} & } FC$.
\end{proof} 

It is clear that a dual proof of this result holds for $\sigma$-s-op-limits. 
Exactly as in \cite{PIE}, we have as a corollary:

\begin{theorem}
A 2-category has all $\sigma$-s-limits (resp.  $\sigma$-s-op-limits) with a PIE indexing pair if and only if it has all
products, inserters, and equifiers. A 2-functor between 2-categories which have these limits preserves all PIE-indexed $\sigma$-s-limits (resp.  $\sigma$-s-op-limits) if and only if it preserves products,
inserters, and equifiers. \qed
\end{theorem}

\noindent Note that from Proposition \ref{prop:yy} and \cite[Coro. 3.3]{PIE}, it follows:

\begin{proposition} \label{prop:zz}
Let $(\A,\Sigma$) be PIE. If $\A \mr{W} \Cat$ satisfies that $W$-weighted limits are the same as $\sigma$-s-limits (resp. $\sigma$-s-oplimits), then $W$ is a PIE weight.
\end{proposition}

By Propositions \ref{prop:xx} and \ref{prop:zz}, we have that the assignations between weights and indexing pairs given in \cite[Th. 14, 15]{S} restrict to PIE weights and PIE indexing pairs. This justifies the abuse we do in the following Section, where we intend ``PIE limits" to mean indistinctly PIE-weighted limits or $\sigma$-s-(op-)limits with a PIE indexing pair.

\section{PIE limits in the $2$-categories of weak algebra morphisms} \label{sec:PIElimitsofalgebras}
 
We fix an arbitrary family $\Omega$ of $2$-cells of a $2$-category $\cc{K}$, closed under horizontal and vertical composition, and containing all the  identity $2$-cells. 
We briefly recall from \cite[\S 2]{S.lifting} how to define the 2-categories of weak algebra morphisms with respect to $\Omega$.
By considering the families $\Omega_\gamma$, $\gamma=s,p,\ell$ consisting of the identities, the invertible $2$-cells and all the $2$-cells of $\cc{K}$, we recover the $2$-categories of algebra morphisms usually considered.
  
We consider a $2$-monad $T$ on $\cc{K}$, and strict algebras of $T$.
 A weak morphism, or \mbox{$\omega$-morphism} (with respect to $\Omega$), $(A,a) \xr{(f,\overline{f})} (B,b)$ is given by $A \mr{f} B$ and $\vcenter{\xymatrix@C=1.5pc@R=1.5pc{TA \ar[r]^{Tf} \ar[d]_a \ar@{}[dr]|{\Downarrow \overline{f}} & TB \ar[d]^b \\
				  A \ar[r]_{f} & B}}$
				  subject to the usual coherence conditions. 
				  The 2-cells considered are the usual algebra 2-cells, and in this way we have a $2$-category $T$-$Alg_{\omega}^{\Omega}$ of \mbox{$T$-algebras} and $\omega$-morphisms, and a forgetful $2$-functor $T$-$Alg_{\omega}^{\Omega} \xr{U_{\omega}^{\Omega}} \cc{K}$. 	

\begin{notation}
 For a $2$-functor $\A \mr{\overline{F}} T$-$Alg_{\omega}^{\Omega}$, 
 we denote $F = U_{\omega}^{\Omega} \overline{F}$ and, for each arrow $f$ of $\A$, we denote  
 $\overline{F}(f) = (F(f), \overline{F(f)})$.
\end{notation}

\begin{definition} \label{de:preserve}
 Let $\Omega'$ be another family of $2$-cells of $\cc{K}$. We say that a family of morphisms $L \mr{p_i} A_i$ in $T$-$Alg_{\omega}^{\Omega}$ (jointly) detects $\Omega'$-ness if, for any other morphism $Z \mr{z} L$ in $T$-$Alg_{\omega}^{\Omega}$, if all the compositions $p_i z$ are $\omega$-morphisms with respect to $\Omega'$, then so is $z$. 
 
 If $\Omega' = \Omega_s$, we say "detect strictness". If 
$\Omega' = \Omega_p$, we say "detect pseudoness".\end{definition} 

We give now the main result of this article. 
Note that, considering items 1 to 3 in Example \ref{ex:pielimits}, we get Propositions 4.2 to 4.4 in \cite{S.lifting}.
 
 \begin{theorem} \label{teo:main}
Let $(\A,\Sigma)$ be PIE, and let a $2$-functor $\A \mr{\overline{F}} T$-$Alg_{\omega}^{\Omega}$.
 We assume that $\overline{F(f_A)}$ is an invertible 2-cell for each $A \in \A$. 
  If $\sosopLim{A \in \A}{FA}$ exists in $\K$ and is $\Omega$-compatible, then $\sosopLim{A \in \A}{\overline{F}A}$ exists in $T$-$Alg_{\omega}^{\Omega}$ and is preserved by $U_{\omega}^{\Omega}$. 
 In other words, the forgetful $2$-functor $U_{\omega}^{\Omega}$ creates this type of $\sigma$-$s$-op-limits. 
  The family of projections $\{\pi_{A_0}\}_{A_0 \in \A_0}$ of this limit are strict, and they jointly detect $\Omega'$-ness for any family $\Omega'$ such that $\sosopLim{A \in \A}{FA}$ is also $\Omega'$-compatible.
 \end{theorem}

\begin{proof}
 Denote $L = \sosopLim{A \in \A}{FA}$. We construct first a (op-)lax cone $\theta = (\theta_A, \theta_f)$ with vertex $TL$, where $\theta_A = a T(\pi_A)$ and  $\theta_f = (\overline{Ff} T\pi_A) (a T\pi_f)$ (see the middle part of the diagram \eqref{eq:diagconop} below). 
 Axioms {\bf LC0-2} are all easy checks (for the last one use that $F\gamma$ is an algebra 2-cell).
 We consider also for each $A \in \A$ the arrow $\mu_A$ defined as the composition $TL \mr{T\pi_{A_0}} TFA_0 \mr{a_0} FA_0 \mr{Ff_A} FA$ and the 2-cell $\alpha_A = T\pi_{A_0} \overline{Ff_A} : \theta_A \Rightarrow \mu_A$. By Lemma \ref{lema:nuevocono} we have a lax cone $\mu$ for which each structural 2-cell $\mu_f$ is given by the composition
 
  \begin{equation} \label{eq:diagconop}
\vcenter{ \xymatrix{&& TL \ar@/_2ex/[dll]_{T\pi_{A_0}} \ar@/^2ex/[drr]^{T\pi_{B_0}} 
 \ar[rd]^{T \pi_B} \ar[dl]_{T \pi_A} \ar@{}[d]|{\substack{T\pi_f \\ \Leftarrow}} \\
 TFA_0 \ar@{}[dr]|{\substack{\overline{Ff_A} \\ \Leftarrow}} \ar[r]^{TFf_A} \ar[d]_{a_0} &
 TFA \ar[rr]^{TFf} \ar[d]_a \ar@{}[drr]|{\substack{\overline{Ff} \\ \Leftarrow}} && TFB \ar[d]^b \ar@{}[dr]|{\substack{\overline{Ff_B}^{-1} \\ \Leftarrow}} & TFB_0 \ar[l]_{TFf_B} \ar[d]^{b_0} \\
				  FA_0 \ar[r]_{Ff_A} & FA \ar[rr]_{Ff} && FB & FB_0 \ar[l]^{Ff_B}}}
\end{equation}
 
It is for checking that this lax cone is in fact a $\sigma$-s-cone (that is, axiom {\bf $\sigma$sC}) that we will use the full strength of the PIE hypothesis: if $f$ as above is in $\Sigma$, then by the unicity of the pair $(B_0,f_B)$ we have $A_0 = B_0$ and $f f_A = f_B$, so that \eqref{eq:diagconop} is an identity 2-cell. 

From the one-dimensional universal property of the limit $\sosopLim{A \in \A}{FA}$, we have a unique $TL \mr{l} L$ such that $\pi_A l = \mu_A$ and $\pi_f l = \mu_f$ for every $A,f$. The usual $T$-algebra axioms for $L$ (see for example \cite[(1.2)(1.3)]{K2dim}) follow from those of the $FA_0$ (for all $A_0 \in \A_0$) using Corollary \ref{coro:projjm} and the naturality of the unit and the multiplication of $T$. 
For $A \in \A$, we have the 2-cell $\overline{\pi_A} = \alpha_A$ and in this way $(\pi_A,\overline{\pi_A})$ is an algebra morphism and the equality $\pi_f l = \mu_f$ expresses that the $\pi_f$ are algebra $2$-cells, thus $\pi$ is a $\sigma$-s-cone in $T$-$Alg_{\omega}^{\Omega}$ which we will show is the $\sigma$-s-limit.
Note that for $A_0 \in \A_0$, by unicity we have $f_{A_0} = id$ and thus $\pi_{A_0}$ is a strict morphism.

To show the one-dimensional universal property of this limit, consider another $\sigma$-s-cone $E \xr{\alg{h_A}} FA$, $\alg{hb} \Mr{h_f} \alg{Ff} \alg{ha}$. We need to show that there is a unique $E \xr{\alg{h}} L$ such that $\alg{h_A} = (\pi_A,\overline{\pi_A}) \alg{h}$ for each $A$, that is $\pi_A h = h_A$, $\pi_f h = h_f$ and $(\pi_A \overline{h})(\overline{\pi_A} Th) = \overline{h_A}$. By the universal property in $\K$, there exists a unique $h$ satisfying the first two of these equalities. 
Noting that $h_* e = \pi_* h e$, $\mu_* Th = \pi_* l Th$, it remains thus to show that there is a unique 2-cell $he \Mr{\overline{h}} l Th$ in $\Omega$ such that $\pi_* \overline{h}$ equals the composition $\mu_* Th \mr{\overline{\pi_A}^{-1} Th} \theta_* Th \mr{\overline{h_A}} h_* e$, which will follow by the $\Omega$-compatibility hypothesis once we show that this composition is a modification. The fact that the $\overline{h_A}$ are so is equivalent to the fact that each $h_f$ is an algebra $2$-cell (see \cite[Th. 5.1]{S.lifting} for details), and so we conclude by Lemma \ref{lema:nuevocono} (recall that $\overline{\pi_A} =  \alpha_A$). 
Recall (Remark \ref{rem:A0compatible}) that a PIE limit is $\Omega'$-compatible if and only if it is $\A_0$-$\Omega'$-compatible, so if this is the case for $L$ then $\overline{h}$ is in $\Omega'$ when each $\overline{h_{A_0}}$ is, giving the last assertion of the theorem (recall that each $\pi_{A_0}$ is a strict morphism, so $\overline{\pi_{A_0}} = id$). 
The coherence conditions for $\overline{h}$ follow from those of the $\overline{h_{A_0}}$ using the last statement in Corollary \ref{coro:projjm}.

Now, for the $2$-dimensional universal property in $T$-$Alg_{\omega}^{\Omega'}$, we consider two cones $h_*,g_*$ with vertex $E$ and a modification with components $h_A \Mr{\beta_A} g_A$. By the universal property of the limit in $\cc{K}$, we have the desired 2-cell $h \Mr{\alpha} g$, and it suffices to check that it is an algebra 2-cell. Again, this follows using the last statement in Corollary \ref{coro:projjm} since each $\beta_{A_0}$ is by hypothesis an algebra 2-cell.
\end{proof}

It is now known (\cite[\S 6.4]{LSAdv}) that PIE limits are the only ones that can be lifted to all the $T$-$Alg_p$ 2-categories, so  this Theorem is in a sense as general as such a lifting result can be (see Corollary \ref{coro:paraalgp}). When compared to the result in op. cit., we observe that the proof of Theorem \ref{teo:main} is much more direct, and that it allows to recover the extra strictness property of the distinguished projections. Also, in the lax case it is significantly stronger and it has many previous results as particular cases (see below).

\smallskip

We consider first the 2-category $T$-$Alg_p$ of pseudo morphisms of algebras (originally called morphisms of algebras in \cite{K2dim}). 
Putting $\Omega = \Omega_p$, $\Omega' = \Omega_s$, all the hypothesis of Theorem \ref{teo:main} are immediately satisfied and we have:

\begin{corollary} \label{coro:paraalgp}
The forgetful $2$-functor $T$-$Alg_p \mr{} \cc{K}$ creates PIE limits. 
For any PIE indexing pair $(\A,\Sigma)$ of this limit, the family of projections $\{\pi_{A_0}\}_{A_0 \in \A_0}$ are strict, and they jointly detect strictness. \qed
\end{corollary}

When applied to the items 1 to 5 in Example \ref{ex:pielimits}, we obtain the Propositions 2.1 to 2.5 in \cite{K2dim}. This procedure yields not only a unified proof of these propositions, but also a slight strengthening: since for applying this Corollary we don't depend on the construction on the limit we want to lift in terms of products, inserts and equifiers, the category $\K$ is not required to have these limits, but only the one that is lifted (see also  \cite[Remark 2.8]{K2dim}).

The result in \cite[Theorem 2.6]{K2dim}, regarding the lifting of lax and pseudo limits can also be obtained from this Corollary, recalling that these are PIE limits (see \cite[p.45]{PIE}). 
In the expression in op. cit. of a lax or pseudo limit weighted by $W$ as a PIE weighted limit, we note that the objects which will define our $\A_0$ are given by the pairs $(x,A)$ with $x \in WA$, so the projections that are strict and detect strictness are exactly the same as in \cite[Theorem 2.6]{K2dim}.
We note that the lifting of the more general $\sigma$-limits is obtained in \cite[Th. 5.1]{S.lifting} using as in the present paper their conical expression but with a proof simpler to the one of Theorem \ref{teo:main} (for $\sigma$-limits, the cone $\theta$ in the proof of \ref{teo:main} suffices and so $\mu$ isn't needed). A follow-up paper to \cite{DDS1} with some further results for $\sigma$-limits is under preparation in which we plan in particular to construct the weights that give $\sigma$-limits as strict limits. Since by \cite[Th. 5.1]{S.lifting} and the result in \cite[\S 6.4]{LSAdv} mentioned above it follows that these weights have to be PIE, that construction would also allow to apply the Corollary above in this case.

\smallskip

We consider now the 2-category $T$-$Alg_\ell$ of lax morphisms of algebras. 
Putting $\Omega = \Omega_\ell$ and $\Omega' = \Omega_s$ (or $\Omega_p$), we have

\begin{corollary} \label{coro:paralaxalg}
Let $(\A,\Sigma)$ be PIE, and let a $2$-functor $\A \mr{\overline{F}} T$-$Alg_{\ell}$.
 We assume that $\overline{F(f_A)}$ is an invertible 2-cell for each $A \in \A$. 
Then the forgetful $2$-functor $U_\ell$ creates $\ssopLim{A \in \A}{\overline{F}A}$. 
The family of projections $\{\pi_{A_0}\}_{A_0 \in \A_0}$ are strict, and they jointly detect strictness (and pseudoness). \qed
\end{corollary}

Unlike the lax case of \cite[Prop. 6.9]{LSAdv}, here we have several results as particular cases. 
Considering items 2, 3 and 6 in Example \ref{ex:pielimits}, we obtain the Propositions 4.3, 4.4 and 4.6 (and therefore Theorem 3.2) of \cite{Llax} in the strong sense of Section 6 therein. 
Also, from Example (5) in \cite[p. 40]{PIE} it follows Proposition 4.5 in \cite{Llax}.
We make the remark that these results don't follow in general from the expression of PIE limits as $\sigma$-s-oplimits that comes from Proposition \ref{prop:weightedcomoconical}, since in this case we would have stronger hypothesis (this is similar to the case analyzed in \cite[Ex. 5.3]{S.lifting}). 

The family of arrows $\{ f_A \}_{A \in \A}$ gives precisely the arrows of the diagram which are required to be pseudo morphisms for the limit to be lifted in each of the cases above.  
In particular, this provides an explanation to {\em why} such a hypothesis is required in each of these cases, one that 
doesn't depend on the construction of the limit to be lifted in terms of other limits but rather on its {\em presentation} as a $\sigma$-s-limit.
Also, as in the case $\Omega = \Omega_p$ above, for this application of Theorem \ref{teo:main} we don't require $\cc{K}$ to have any other limit than the one we lift.

On the other hand, it seems that the op-lax limits considered in \cite[Th. 4.8]{Llax}, \cite[Cor. 5.9]{S.lifting} can't be lifted to $T$-$Alg_\ell$ using this Corollary, unless the diagram is in $T$-$Alg_p$. Of course, when this is the case by Corollary \ref{coro:paralaxalg} we have more generally that all PIE limits lift (cf. \cite[Prop. 4.1]{Llax}):

\begin{corollary}
The inclusion $T$-$Alg_p \mr{} T$-$Alg_\ell$ preserves all PIE limits.
The family of projections $\{\pi_{A_0}\}_{A_0 \in \A_0}$ 
of such limits in $T$-$Alg_\ell$ are strict and jointly detect strictness. \qed
\end{corollary}

\bibliographystyle{unsrt}

\end{document}